\documentclass[11pt]{article}
\usepackage{amsmath}
\usepackage{amscd}
\usepackage{amsthm}
\usepackage{amssymb} \usepackage{latexsym}
\usepackage{eufrak}
\usepackage{euscript}
\usepackage[width=19cm, height=24cm]{geometry}
\usepackage{epsfig}
\usepackage{tikz}
\usepackage{graphics}
\usepackage{array}
\usepackage{enumerate}
\usepackage{authblk}
\theoremstyle{theorem}
\newtheorem{thm}{Theorem}[section]
\theoremstyle{corollary}

\theoremstyle{lemma}
\newtheorem{lemma}{Lemma}[section]
\theoremstyle{definition}
\newtheorem{definition}{Definition}[section]
\theoremstyle{proposition}
\newtheorem{proposition}{Proposition}[section]
\theoremstyle{proof}
\newtheorem{exm}{Example}[section]

\theoremstyle{remark}

\newcommand{\bel}[1]{\begin{equation}\label{#1}}

\newcommand{\be}{\begin{equation}}

\newcommand{\ba}{\begin{eqnarray}}
\newcommand{\ea}{\end{eqnarray}}

\newcommand{\qe}{\end{equation}}

\begin{document}
\title{Spectra of general hypergraphs}

\author[1, 2]{\rm Anirban Banerjee}
\author[1]{\rm Arnab Char}
\author[1]{\rm Bibhash Mondal}

\affil[1]{Department of Mathematics and Statistics}
\affil[2]{Department of Biological Sciences}
\affil[ ]{Indian Institute of Science Education and Research Kolkata}
\affil[ ]{Mohanpur-741246,  India}
\affil[ ]{\textit {{\scriptsize anirban.banerjee@iiserkol.ac.in, ac13ms134@iiserkol.ac.in,  bm12ip022@iiserkol.ac.in}}}

\maketitle

\begin{abstract}

Here, we show a method to  reconstruct connectivity hypermatrices of a general hypergraph (without any self loop or multiple edge) using tensor. We also study the   different spectral properties of these hypermatrices and find that  these properties are similar for graphs and uniform hypergraphs. The  representation of a connectivity hypermatrix
that is proposed here can be very useful for the further development in spectral hypergraph theory.
\end{abstract}


\section{Introduction}
% 3 para
% Introduction
% Survey / recent work
% Our work

Spectral graph theory has a long history behind its development.  In spectral graph theory, we analyse the eigenvalues of a connectivity matrix which is uniquely defined on a graph. Many researchers have had a great interest to study the eigenvalues of different connectivity matrices, such as, adjacency matrix, Laplacian matrix, signless Laplacian matrix, normalized Laplacian matrix, etc. Now, a recent trend has been developed to explore spectral hypergraph theory. Unlike in a graph, an edge of a hypergraph can be constructed with more than two vertices, i.e., the edge set of a hypergraph is the subset of the power set of the vertex set of that hypergraph \cite{Voloshin}. Now, one of the main challenges is to uniquely represent a hypergraph by a connectivity hypermatrix or by a tensor, and vice versa. It is not trivial for a non-uniform hypergraph, where the cardinalities of the edges are not the same. 
%%%%%%%%%%%%%%%%%%%%%%%%%%%%%%%%%%%%%%%
Recently, the study of the spectrum of uniform hypergraph becomes popular. 
In a ($m$-) uniform hypergraph, each edge contains the same, ($m$), number of vertices. Thus an $m$-uniform hypergraph of order $n$ can be easily represented by an $m$ order $n$ dimensional connectivity hypermatrix (or tensor). 
%% Survey / recent work on different hypermatrics
%\subsubsection{survey}
In \cite{Cooper2012}, the results  on the spectrum of adjacency matrix of a graph are extended for uniform hypergraphs by using characteristic polynomial.  Spectral properties of adjacency uniform hypermatrix  are deduced from matroids in \cite{Pearson2015}.
%%%%%%%%%
In 1993, Fan Chung   defined Laplacian of a uniform hypergraph by considering various homological  aspects of hypergraphs and studied the eigenvalues of the same \cite{chung1993}. In \cite{Hu2013_2, Hu2014, Hu2015, Qi2014, Qi2014_2}, different spectral properties of  Laplacian and signless Laplacian of a uniform hypergraph, defined by using tensor, have been studied.
In 2015, Hu and Qi introduced the normalized  Laplacian of a uniform hypergraph and analyzed its spectral properties \cite{Hu2013}. 
%%%%%%%%%
The important tool that has been used in spectral hypergraph theory is tensor. In 2005, Liqun Qi introduced the different eigenvalues of a real supersymmetric tensor \cite{Qi2005}. The various properties of the eigenvalues of a tensor have been studied in \cite{Chang2008, Chang2009, Li2013, Ng2009, Shao2013, shaoshan2013,  Yang2010, YangYang20101}. 
%%%%%%%%% %%%%%% 
 
 But, still the challenge remains to come up with a mathematical framework to construct a connectivity hypermatrix for a non-uniform hypergraph, such that, based on this connectivity hypermatrix the spectral graph theory for a general hypergraph can be developed. Here, we propose a unique representation of a general hypergraph (without any self loop or multiple edge) by connectivity hypermatrices, such as, adjacency hypermatrix, Laplacian hypermatrix, signless Laplacian hypermatrix, normalized Laplacian hypermatrix and analyze the different spectral properties of these matrices. These properties are very similar with the same for graphs and uniform hypergraphs. Studying the spectrum of a uniform hypergraphs could be considered as a special case of the  spectral graph theory of general hypergraphs.

%%%%%%%%% %%%%%% 
\section{Preliminary}
Let $\mathbb{R}$ be the set of real numbers. We consider an  $m$ order $n$  dimensional hypermatrix  $A$  having $n^m$ elements from $\mathbb{R}$, 
 where 
 $$A=(a_{i_{1}, {i_{2}}, \dots, {i_{m}}}),   a_{i_{1}, {i_{2}}, \dots, {i_{m}}} \in \mathbb{R} \text{ and } 1\leq i_{1}, i_{2}, \dots, i_{m} \leq n.$$
 Let $x=(x_{1}, x_{2}, \dots, x_{n})\in \mathbb{R}^n$. If we write $x^{m}$ as an $m$ order $n$ dimension hypermatrix with $(i_1,  i_2,  \dots,   i_m)$-th entry $x_{i_1}  x_{i_2}\dots   x_{i_m}$,  then $Ax^{m-1}$, where the multiplication is taken as tensor contraction over all indices, is an $n$ tuple whose $i$-th component is 
 $$ \sum_{i_{2}, i_{3}, \dots,  i_{m}=1}^{n} a_{i i_{2}i_{3}\dots i_{m}} x_{i_{2}}x_{i_{3}} \dots x_{i_{m}} .$$
 
\begin{definition}
  Let $A$ be a nonzero hypermatrix. A pair $(\lambda,  x) \in \mathbb{C} \times (\mathbb{C}^{n}\setminus \{0\})$  is called eigenvalue and  eigenvector (or simply an eigenpair) if they satisfy the following equation
  $$Ax^{m-1}=\lambda x^{[m-1]}.$$ Here,  $x^{[m]}$ is a vector with $i$-th entry $x^{m}_i$.
   We call $(\lambda,  x)$  an $H$-eigenpair (i.e.,  $\lambda$ and $x$ are called $H$-eigenvalue and   $H$-eigenvector,  respectively) if they are both real. An $H$-eigenvalue $\lambda$ is called $H^+  (H^{++})$-eigenvalue  if the corresponding eigenvector 
   $x \in \mathbb{R}_+^n $
   $(\mathbb{R}_{++}^n)$.
\end{definition}

 \begin{definition}
  Let $A$ be a nonzero hypermatrix. A pair $(\lambda,  x) \in \mathbb{C} \times (\mathbb{C}^{n}\setminus \{0\})$  is called an $E$-eigenpair (where $\lambda$ and $x$ are called $E$-eigenvalue and  $E$-eigenvector,  respectively)
  if they satisfy the following equations 
  $$Ax^{m-1}= \lambda x,$$
  $$ \sum_{i=1}^{n} x_{i}^2 =1.$$
  We call $(\lambda,  x)$ a $Z$-eigenpair if both of them are real.
 \end{definition}
From the above definitions it is clear that,  a constant multiplication of an eigenvector is also an eigenvector corresponding to an $H$-eigenvalue,  but,  this is not always true for $E$-eigenvalue and $Z$-eigenvalue.
%%%%%
Now, we recall some results that are used in the next section.
\begin{thm}[\cite{Qi2005}]
The eigenvalues  of $A$ lie in the union of $n$ disks in $\mathbb{C}$. These n disks have the diagonal elements of the supersymmetric tensor as their centers, and the sums of the
absolute values of the off-diagonal elements as their radii.
\end{thm}
The above theorem helps us to  bound  the eigenvalues of a tensor.
\begin{lemma}\label{co-spectral}
Let $ A$ be an $m$ order and $n$ dimensional tensor and $D=diag(d_1,\dots ,d_n)$ be a positive diagonal matrix. Define a new tensor $${B}=A.D^{-(m-1)}.\overbrace{D \dots D}^{m-1} $$ with the entries $$B_{i_{1}i_{2} \dots i_{m}}=A_{i_{1}i_{2} \dots i_{m}}d_{i_1}^{-(m-1)}d_{i_{2}} \dots d_{i_{m}}.$$

Then $A$ and $B$ have the same $H$-eigenvalues.
\end{lemma}
\begin{proof}
From the remarks of lemma (3.2) in \cite{Yang2010}.
\end{proof}
Some results of spectral graph theory\footnote{For different spectral properties of a graph see \cite{bapat2010, chung1997}}  also hold for general hypergraphs.
If $\lambda$ is any eigenvalue of an adjacency matrix of a graph $G$ with the maximal degree $\Delta$, then
$\lambda\leqslant\Delta$.
For a $k$-regular graph $k$ is the maximum eigenvalue with a constant eigenvector of the adjacency matrix of that graph. 
 If $\lambda$ and $\mu$  are  the eigenvalues of the adjacency matrices, represent the graphs $G$ and $H$,  respectively,  then $\lambda + \mu$ is also an eigenvalue of the same for  $G \times H$, the  cartesian product of   $G$ and $H$.
All the eigenvalues of a Laplacian matrix of a graph are nonnegative and a very rough upper bound of these eigenvalues is $2\Delta$, whereas, any eigenvalue of a normalized Laplacian matrix of a graph lies in the interval $[0, 2]$. 
 Zero is always an eigenvalue for both, Laplacian and normalized Laplacian matrices, of a graph, with a constant eigenvector. 
If $\mathbb{A}$ and $\mathcal{L}$ are the normalized adjacency matrix and normalized Laplacian matrix, respectively, of a graph (such that $ \mathcal{L} = I - \mathbb{A}$)  then the spectrum of $\mathbb{A}$,    $\sigma(\mathbb{A})=1-\sigma(\mathcal{L})$.
If $M$ is a connectivity matrix of a graph with $r$ connected components then $\sigma(M)=\sigma(M_1)\cup\sigma(M_2)\dots\cup\sigma(M_r)$, where $M_{i}$ is the same connectivity matrix corresponding to the component $i$.

%%%%%%%%%%%%%%%%%%
 
\section{Spectral properties of general hypergraphs}

\begin{definition}
 A (general) hypergraph $G$ is a pair $G= (V, E)$ where $V$ is a set of elements called vertices,  and $E$ is a set of non-empty subsets of $V$ called  edges. Therefore,  E is a subset of $\mathcal{P}(V) \setminus\{\emptyset\}$,  where $\mathcal{P}(V)$ is the power set of $V$.
\end{definition}

\begin{exm}\label{Examp:Hypergraph}
 Let $G=(V, E)$,  where $V=\{ 1,  2,  3,  4,  5 \}$ and $E= \big\{ \{ 1 \},  \{ 2,  3 \},  \{ 1,  4,  5 \} \big\}$. Here,  $G$ is a hypergraph of $5$ vertices and $3$ edges.
\end{exm}

%%%%%%%%%%%%%%%%%%
\subsection{Adjacency hypermatrix and eigenvalues}
 \begin{definition}
  Let $G=(V, E)$ be the hypergraph where $V= \{ v_{1},  v_{2},  \dots,  v_{n} \}$ and $E = \{ e_{1},  e_{2},  \dots,  e_{k} \}$. Let $m=max\{|e_{i}| : e_{i}\in E\}$ be the maximum cardinality of edges,  $m.c.e(G)$,  of $G$.
  Define the adjacency hypermatrix of $G$ as 
  $$\mathcal{A}_G = (a_{i_{1}i_{2}\dots i_{m}}),  \text{  } 1\leq i_{1}, i_{2}, \dots, i_{m} \leq n.$$
  For all edges $e=\{ v_{l_{1}},  v_{l_{2}},  \dots, v_{l_{s}}\}\in E$ of cardinality $s \leq m$,   
  $$a_{ p_{1}p_{2}\dots p_{m}}=\frac{s}{\alpha},  \text{ where } \alpha =\sum_{k_{1}, k_{2}, \dots, k_{s} \geq 1,  \sum k_{i}=m}  \frac{m!}{k_{1}! k_{2}!\dots k_{s}!}, $$   and   $p_{1}, p_{2},  \dots,  p_{m}$  chosen in all possible way  from $\{l_{1}, l_{2}, \dots, l_{s}\}$ with at least once
  for each element of the set. The other positions of the hypermatrix are zero\footnote{For a similar construction on uniform multi-hypergraph see \cite{Pearson2014}.}.
 \end{definition}
 
 \begin{exm}
  Let $G=(V, E)$ be a hypergraph in example \ref{Examp:Hypergraph}. Here,  the maximum cardinality of edges is
  $3$. The adjacency hypermatrix of $G$ is $\mathcal{A}_G = (a_{i_{1}i_{2}i_{3}})$,  where $1\leq  i_1,  i_2,  i_3 \leq 5$. Here,  $a_{111}=1,  a_{233}=a_{232}=a_{223}=a_{323}=a_{332}=a_{322}=\frac{1}{3},  a_{145}=a_{154}=a_{451}=a_{415}=a_{541}=a_{514}=\frac{1}{2}$,  and the other elements of $\mathcal{A}_G$ are zero.
 \end{exm}

 \begin{definition}
  Let $G=(V, E)$ be a hypergraph. The degree, $d(v)$, of a vertex $v \in V$   is the number of edges consist of $v$.
 \end{definition}

  Let $G=(V, E)$ be a hypergraph,  where $V=\{v_{1},  v_{2},  \dots,  v_{n}\}$ and $E=\{e_{1}, e_{2}, \dots,  e_{k}\}$. Then,  the  degree of a vertex $v_{i}$ is given by
   $$d(v_{i})= \sum_{i_{2}, i_{3}, \dots,  i_{m}=1}^{n} a_{ii_{2}i_{3}\dots i_{m}}.$$
   
 \begin{definition}
 A hypergraph is called $k$-\textit{regular} if every vertex has the same degree $k$.
 \end{definition}
 Now,  we discuss some spectral properties of $\mathcal{A}_G$ of a hypergraph $G$. Some of these properties are very similar as in general graph (i.e.~for a $2$-uniform hypergraph).
  \begin{thm}\label{HEigUB}
  Let $\mu$ be an $H$-eigenvalue of $\mathcal{A}_G$. Then $|\mu|\leq\Delta$,  where $\Delta$ is the maximum degree of $G$.
   \end{thm}

 \begin{proof}
   Let $G$ be a hypergraph with $n$ vertices and $m.c.e(G)=m$.  Let $\mu$ be an $H$-eigenvalue of $\mathcal{A}_G=(a_{i_{1}i_{2}\dots i_{m}})$ with  an eigenvector $x =( x_{1}, x_{2}, \dots, x_{n})$. Let $x_{p}=max \{ |x_{1}|,  |x_{2}|,  \dots,  |x_{n}| \}$. Without loss of any generality we can assume that $x_{p}=1$.
Now,  $$|\mu|=|\mu x_{p}^{m-1}|=\bigg|\sum_{i_{2}, i_{3}, \dots,  i_{m}=1}^{n} a_{pi_{2}i_{3}\dots i_{m}} x_{i_{2}}x_{i_{3}}\dots x_{i_{m}}\bigg|$$
$$ \leq \sum_{i_{2}, i_{3}, \dots,  i_{m}=1}^{n} |a_{p i_{2}i_{3}\dots i_{m}}| |x_{p}|^{m-1} =d(v_{p})\leq \Delta.$$
\end{proof}
Thus,  for a $k$-regular hypergraph the theorem (\ref{HEigUB}) implies $|\mu| \leq k$. 

\begin{thm}
 Let $G=(V, E)$ be a $k$-regular hypergraph with $n$ vertices. Then,  $\mathcal{A}_G=(a_{i_{1}i_{2}\dots i_{m}})$  has an $H$-eigenvalue $k$.
 \end{thm}
 \begin{proof}
  Since,  $G$ is $k$-regular,  then $d(v_{i})=k$ for all $v_{i}\in V$,  $i\in\{1, 2, 3, ..., n\}$.
 Now,  for a vector $x=(1, 1, 1, \dots,  1)\in \mathbb{R}^{n}$, we have
  
 $$\mathcal{A}_Gx^{m-1}= \sum_{i_{2}, i_{3}, \dots i_{m}=1}^{n} a_{ii_{2}i_{3}\dots i_{m}}=k.$$
 Thus the proof.
 \end{proof}

\begin{thm}
  Let $G=(V, E)$ be a $k$-regular hypergraph  with $n$ vertices. Then,   $\mathcal{A}_G=(a_{i_{1}i_{2}\dots i_{m}})$  has a $Z$-eigenvalue  $k(\frac{1}{\sqrt{n}})^{m-2}$.
  \end{thm}
  \begin{proof}
  The vector $x=(\frac{1}{\sqrt{n}}, \frac{1}{\sqrt{n}}, \dots,  \frac{1}{\sqrt{n}})\in \mathbb{R}^{n}$ satisfies the $Z$-eigenvalue equations for $\lambda =k(\frac{1}{\sqrt{n}})^{m-2}$.
  \end{proof}
 
\begin{thm}
  Let $G$ be a hypergraph with $n$ vertices and maximum degree $\Delta$.  Let $x=(x_{1},  x_{2},  \dots,  x_{n})$ be a $Z$-eigenvector of $\mathcal{A}_G=(a_{i_{1}i_{2}\dots i_{m}})$  corresponding to an eigenvalue $\mu$. If $x_{p}=max  \big\{|x_{1}|, |x_{2}|, \dots,  |x_{n}| \big\}$,  then $|\mu|\leq \frac{\Delta}{x_{p}}$.
 \end{thm} 
  \begin{proof}
   The $Z$-eigenvalue equations of $\mathcal{A}_G$ for $\mu$ and $x$ are $Ax^{m-1}=\mu x$,  and $\sum x_{i}^{2} =1$. Therefore,  $|x_{i}|\leq 1$,  for all $i=1, 2, 3, \dots,  n$. Now,  $$|\mu| |x_{j}|= \bigg|\sum_{i_{2}, i_{3}, \dots,  i_{m}=1}^{n} a_{ji_{2}i_{3}\dots i_{m}} x_{i_{2}}x_{i_{3}}\dots x_{i_{m}} \bigg|, $$ which implies  $|\mu| |x_{j}|\leq d(j)\leq \Delta,   \forall j=1, 2, 3, \dots, n$. Therefore,  $|\mu|\leq \frac{\Delta}{x_{p}}$.
  \end{proof}

\begin{definition}
A hypergraph $H=(V_{1}, E_{1})$ is said to be a spanning subhypergraph of a hypergraph $G=(V, E)$,  if $V=V_{1}$ and $E_{1}\subseteq E$.
\end{definition}

\begin{thm}
 Let $G=(V, E)$ be hypergraph. Let $H=(V^{'}, E^{'})$ be a   subhypergraph of $G$, such that, $ m.c.e(G)=m.c.e(H)$ be even. Then,  $\mu _{max}(H)\leq \mu _{max}(G)$,  where $\mu_{max} $ is the highest $Z$-eigenvalue of the corresponding adjacency  hypermatrix.
 \end{thm}
 \begin{proof}
 Let $|V|=n$, $|V^{'}|=n^{'}$ ($\leq n$) and $m.c.e(G)=m.c.e(H)=m.$ 
  Now, 
  \begin{align*}  
  \mu _{max}(H)& = max _{||x||=1} x^{t} \mathcal{A}_{H} x^{m-1} \text{ (by using lemma (3.1) in \cite{Li2013})}\\
   & =    max _{||x||=1} \bigg(\sum_{i_{1},  i_{2},  \dots,  i_{m}=1}^{n^{'}}   a^H_{i_{1}i_{2}...i_{m}} x_{i_{1}} x_{i_{2}}\dots x_{i_{m}}\bigg)\\    
   & =    max _{||x||=1} \bigg(\sum_{i_{1},  i_{2},  \dots,  i_{m}=1}^{n}   a^H_{i_{1}i_{2}...i_{m}} x_{i_{1}} x_{i_{2}}\dots x_{i_{m}}\bigg), \text{ where } a^H_{i_{1}..i_{m}}=x_{i_r}=0 \text{ when } i_r >n^{'}\\ 
   &\leq  \bigg(\sum_{i_{1}, i_{2}, \dots,  i_{m} =1}^{n}  a^{G}_{i_{1}i_{2}\dots i_{m}} x_{i_{1}} x_{i_{2}}\dots x_{i_{m}}\bigg)  \\
  &\leq\mu _{max}(G), 
  \end{align*}
  since  each component of $x$ is nonnegative (by Perron-Frobenious theorem \cite{Chang2008}) and the number of edges of $G$ is greater than or equal to the number of edges of $H$. Hence the proof.
 \end{proof}

\begin{definition}
 Let $G=(V, E)$ be a hypergraph  with $V=\{v_{1},  v_{2},  \dots,  v_{n}\}$,  $E=\{e_{1},  e_{2},  \dots,  e_{k}\}$,  and $m.c.e(G)=m$. Let $x=(x_1, x_2, \dots,  x_n)$ be a vector in $\mathbb{R}^n$ and $p\geq s-1$ be an integer. For an edge $e=\{ v_{l_{1}},  v_{l_{2}},  \dots,  v_{l_{s}}\}$  and a vertex $v_{l_i}$,  we define  $$x^{e/v_{l_i}}_{p} := \sum x_{r_1}x_{r_2}\dots x_{r_p}, $$
where the sum is over $r_{1},  r_{2}, \dots,  r_{p}$ chosen in all possible way  from $\{l_{1},  l_{2},  \dots,  l_{s}\}$,  such that,  all $l_j (j \ne i)$ occur at least once.
Whereas, $$x^{e}_{p} := \sum x_{r_1}x_{r_2}\dots x_{r_p}, $$
where the sum is over $r_{1},  r_{2}, \dots,  r_{p}$  chosen in all possible way  from $\{l_{1},  l_{2},  \dots,  l_{s}\}$ with at least once for each element of the set.

The symmetric (adjacency) hypermatrix $\mathcal{A}_G$ of order $m$ and dimension $n$ uniquely defines a homogeneous polynomial of degree $m$ and in $n$ variables by
  $$F_{\mathcal{A}_G}(x)=\sum_{i_1, i_2, \dots, i_m =1}^{n} a_{i_1 i_2\dots i_m} x_{i_1}x_{i_2}\dots x_{i_m}.$$
\end{definition}
We rewrite the above polynomial as: $$F_{\mathcal{A}_G}(x)= \sum_{e \in E} a_{e}^{G} x^{e}_{m}, $$ where $a_{e}^{G}=\frac{s}{\alpha}$,    $\alpha =\sum_{k_{1}, k_{2}, \dots,  k_{s} \geq 1,  \sum k_{i}=m}  \frac{m!}{k_{1}! k_{2}!\dots k_{s}!}$,  and $s$ is the cardinality of the edge $e$.

\begin{definition}
 Let $G$ and $H$ be two hypergraphs. The Cartesian product,  $G \times H$,  of $G$ and $H$ is defined by the vertex set $V(G \times H)=V(G) \times V(H)$ and the edge set  $E(G \times H)= \big\{ \{v\} \times e : v \in V(G),  e \in E(H)\big\}
 \bigcup \big\{ e \times \{v\} : e \in E(G),  v \in V(H)\big\}.$
\end{definition}

\begin{definition}
Let $G$ be a hypergraph  with the vertex set $V=\{v_{1},  v_{2},  \dots,  v_{n}\}$ and $m.c.e(G)=m$. For an edge 
$e= \{ v_{l_1},  v_{l_2},  \dots,  v_{l_s} \}$ and an integer $r \geq m$, the arrangement  $(v_{p_1}  v_{ p_2} \dots v_{p_r})$ (where $p_1,  p_2,  \dots,  p_r$ are  chosen in all possible way  from $\{ l_{1},  l_{2},  \dots,  l_{s} \}$ with at least once   for each element of the set) represents the edge $e$ in order $r$.
\end{definition}
 
\begin{exm}
Let $G=(V, E)$ where $V= \{1, 2, 3, 4, 5\}$ and $E= \big\{ \{1,  2,  3 \},  \{ 2,  3,  5 \},  \{ 1,  3,  4,  5\}  \big\}$,  then the arrangement $(12233)$ represents the edge $\{1,  2,  3 \}$ in order 5.   $(12123)$ is also a representation of the edge
$\{1,  2,  3 \}$ in order five,  whereas,  $(111123)$ represents  the edge $\{1,  2,  3 \}$ in 6 order.
\end{exm}

Let $G=(V, E)$ be a hypergraph with $m.c.e(G) =m$ and $E_i= \{ e\in E : v_i \in   e\}$. Now,  the $H$-eigenvalue equation for $\mathcal{A}_G$ becomes
 $$\sum _{e\in E_i } a_{e}^{G} x^{e/v_i}_{m-1} =\lambda x_{i}^{(m-1)},  \text{  for all  } i.$$

\begin{thm}
 Let $G$ and $H$ be two hypergraphs with $m.c.e(G) = m.c.e(H)$. If $\lambda$ and $\mu$  are  $H$-eigenvalue for $G$ and $H$,  respectively,  then $\lambda +\mu$ is an  $H$-eigenvalue for $G \times H$.
 \end{thm}
 
 \begin{proof}
 Let $n_1$ and $n_2$ be the number of vertices in $G$ and $ H$,  respectively,  and $m.c.e(G) = m.c.e(H)= m$.
  Let $(\lambda, \textbf{u})$ and $(\mu, \textbf{v})$ be $H$-eigenpairs of $\mathcal{A}_G$ and $\mathcal{A}_H$,  respectively.
Let $\textbf{w}\in\mathbb{C}^{n_{1}n_{2}}$  be a vector with the entries indexed by the pairs $(a, b)\in[n_1]\times[n_2]$,  such that,  $w(a, b)=u(a)v(b)$. Now,  we show that $(\lambda+\mu,   \textbf{w})$ is an $H$-eigenpair of $\mathcal{A}_{G\times H}$.
\begin{align*}
 \sum_{e\in E_{(a, b)}}a_{e}^{G\times H}w^{e/{(a, b)}}_{m-1}
  &=\sum_{\substack{\{a\}\times e \in E_{(a, b)}\\ \mbox{ with } e\in E_b}}a_{e}^{G\times H} w^{\{a\}\times e/(a, b)}_{m-1}
   +\sum_{\substack {e \times \{b\} \in E_{(a, b)}\\ \mbox{ with } e \in E_a}} a_{e}^{G\times H}w^{e \times \{b\}/(a, b)}_{m-1}\\
  &=\sum_{e \in H_b}a_{e}^{G\times H} u^{m-1}(a)v^{e/b}_{m-1}+\sum_{e \in G_a}a_{e}^{G\times H}u^{e/a}_{m-1}v^{m-1}(b)\\
   &=u^{m-1}(a) \sum_{e \in H_b}a_{e}^{ H} v_{m-1}^{e/b}+ v^{m-1}(b) \sum_{e \in E_a}a_{e}^{G}u^{e/a}_{m-1}\\
   &=u^{m-1}(a) \mu v^{m-1}(b)+v^{m-1}(b) \lambda u^{m-1}(a)\\
  &=(\lambda+\mu)w^{m-1}(a, b).
\end{align*}
Hence the proof\footnote{For similar proof on uniform hypergraph see \cite{Cooper2012}.}.
\end{proof}
  
\begin{lemma}
\label{largest z eigenvalue}
 Let $A$ and $B$ be two symmetric hypermatrices of order $m$ and dimension $n$,  where $m$ is even. Then $\lambda_{max}(A+B)\leq\lambda_{max}(A)+\lambda_{max}(B)$,   where $\lambda_{max}(A)$ denotes the largest $Z$-eigenvalue of $A$.
 \end{lemma}
 \begin{proof}
 \begin{align*}
 \lambda_{max}(A+B) & = max _{||x||=1} x^{t}(A+B)x^{m-1} \text{ (by using lemma (3.1) in \cite{Li2013})} \\ 
  & \leq max _{||x||=1} x^{t}Ax^{m-1}+max _{||x||=1} x^{t}Bx^{m-1}\\
 &= \lambda_{max}(A)+\lambda_{max}(B).
  \end{align*}
 \end{proof}

Let $G=(V, E)$ be a hypergraph with the vertex set $V=\{v_{1}, v_{2}, \dots,  v_{n}\}$ and $m=max\{|e_{i}|    : e_{i}\in E\}$. We partition the edge set $E$ as,  $E=E_{1} \cup E_{2} \cup \dots \cup E_{m}$,  where $E_{i}$ contains all the edges of the cardinality $i$ and construct a hypergraph $G_{i}=(V, E_{i})$,  for a nonempty $E_i$.

\begin{definition}
Define the adjacency hypermatrix of $G_{i}$ in $m$ $(> i)$ -order by an $n$  dimensional $m$ order  hypermatrix 
$$\mathcal{A}_{G_{i}}^m = \big(       (a_{G_{i}}^m)_{ p_{1}p_{2}\dots p_{m}} \big),  \text {   } 1 \leq p_{1},  p_{2},  \dots,  p_{m} \leq n, $$ such that,  for  any  $e=\{ v_{l_{1}},  v_{l_{2}},  \dots,  v_{l_{i}} \} \in E_i$, 
$$(a_{G_{i}}^m)_{ p_{1}p_{2}\dots p_{m}}=\frac{i}{\alpha},  \text{ where } 
\alpha =\sum_{k_{1}, k_{2}, \dots, k_{i} \geq 1,  \sum k_{j}=m}  \frac{m!}{k_{1}! k_{2}!\dots k_{i}!}$$   
and   $p_{1}, p_{2}, \dots,  p_{m}$ are chosen in all possible way  from $\{l_{1}, l_{2}, \dots, l_{i}\}$ with at least once
  for each element of the set. The other positions of  $\mathcal{A}_{G_{i}}^m$    are zero. 
  \end{definition}
  Thus,  we can represent a hypergraph $G$,  with $m.c.e(G)=s$,  in higher order $m > s$ by the hypermatrix $\mathcal{A}_{G}^m$. Clearly,  all the eigenvalue equations show that the eigenvalues of $\mathcal{A}_{G}^{m_1}$ and $\mathcal{A}_{G}^{m_2}$ are not equal for $m_1 \ne m_2$.

\begin{thm}\label{EdgePartitionTheo}
Let $G=(V, E)$ be a hypergraph and $m.c.e(G)=m$ be even. Then
$\lambda _{max}(\mathcal{A}_{G})\leq \sum _{i=1}^m\lambda _{max}(\mathcal{A}_{G_i}^m)$,  where $\lambda_{max}({A})$ is the largest $Z$-eigenvalue of $A$.
 \end{thm}
\begin{proof}
Since $\mathcal{A}_{G}=\sum _{i=1}^{m} \mathcal{A}_{G_i}^m$,  the proof follows from the lemma ($\ref{largest z eigenvalue}$).
\end{proof}
Moreover,  the theorem (\ref{EdgePartitionTheo}) implies $\lambda _{max}(\mathcal{A}_{G})\leq \sum _{i=1}^m n_i  \lambda _{max}(\mathcal{A}_{i}^m)$,   where $n_{i}$ is the number of edges of cardinality $i$ and $\mathcal{A}_{i}^m$ is the adjacency hypermatrix in $m$-order of a hypergraph contains a single edge of cardinality $i$. 

%%%%%%%%%%%%%%%
\subsection{Laplacian hypermatrix and eigenvalues}

\begin{definition}
  Let $G=(V, E)$ be a (general) hypergraph without any isolated vertex where $V= \{ v_{1},  v_{2},  \dots,  v_{n} \}$ and $E = \{ e_{1},  e_{2},  \dots,  e_{k} \}$. Let $m.c.e(G)=m$. We define the Laplacian hypermatrix, $L_{G}$, of $G=(V, E)$ as $L_{G}=D_{G}-\mathcal{A}_{G}=(l_{i_{1}i_{2}\dots i_{m}}),  1\leq i_{1}, i_{2}, \dots, i_{m} \leq n,$
where $D_{G}=(d_{i_{1}i_{2}\dots i_{m}})$ is the $m$ order $n$ dimensional diagonal hypermatrix with $d_{ii\dots i}=d(v_i)$ and others  are zero.
The signless Laplacian of $G$ is defined as $L_{G}=D_G+\mathcal{A}_{G}.$
  \end{definition}

Let $G=(V, E)$ be a hypergraph with $m.c.e(G)=m$. For any edge $e=\{ v_{l_{1}},  v_{l_{2}},  \dots,  v_{l_{s}}\}$,  we define a homogeneous polynomial of degree $m$ and in $n$ variables by

$$
L(e)x^m = \sum_{j=1}^s x_{i_j}^m - \frac{s}{\alpha}x_m^e \text{  } (s\leq m).
$$
\begin{proposition}\label{inequality}
 $ \sum_{j=1}^s x_{i_j}^m \geq \frac{s}{\alpha}x_m^e$  $(x_{i_j}\in \mathbb{R_{+}})$.
\end{proposition}
\begin{proof}
 $x_m^e$ is the sum of all possible terms,  $x_{i_1}^{k_1} x_{i_2}^{k_2}\dots x_{i_s}^{k_s}$ (where $\sum k_i = m$ and $k_i \geq 1$) $$\text{ where } \alpha =\sum_{k_{1}, k_{2}, \dots, k_{s} \geq 1,  \sum k_{i}=m}  \frac{m!}{k_{1}! k_{2}!\dots k_{s}!}, $$with some natural coefficient. Now,  by applying AM-GM inequality on $k_1 x_{i_1}^{m},  k_2 x_{i_2}^{m}, \dots,  k_s x_{i_s}^{m}$ we get
 \begin{equation}\label{lemma:equation}
  \frac{1}{m} \sum_{j=1}^s k_j x_{i_j}^m \geq x_{i_1}^{k_1} x_{i_2}^{k_2}\dots x_{i_s}^{k_s}.
 \end{equation}
If we apply (\ref{lemma:equation}) for each term of $x_m^e$ and take the sum,  we get
$$
 \frac{\alpha}{s} \sum_{j=1}^s x_{i_j}^m \geq x_m^e.
$$
\end{proof}
Many properties of Laplacian and signless Laplacian tensors are discussed in \cite{Qi2014}. Now we show that some of the results in general graph are also true for non-uniform (general) hypergraph.
Note that, here, $L$ is co-positive tensor since, $Lx^m=\sum _{e\in E}L(e)x^m \geq 0$ for all $x \in \mathbb{R}_{+}^n $.
\begin{thm}\label{positivity}
 Let $G=(V, E)$ be a general hypergraph. Let $L=(l_{i_{1}i_{2}\dots i_{m}})  \text{ where } 1\leq i_{1}, i_{2}, \dots, i_{m} \leq n,$ be the Laplacian hypermatrix of $G$. Then $0\leq \lambda \leq 2\Delta,$ where $\lambda$ is an $H$-eigenvalue of $L$.
\end{thm}
\begin{proof}
 For a vector $y=(\frac{1}{n^{\frac{1}{m}}},\frac{1}{n^{\frac{1}{m}}},\dots, \frac{1}{n^{\frac{1}{m}}}), Ly^m=0$.
 Since $L$ is a co-positive tensor, thus $\textit{min}\{Lx^m :x \in \mathbb{R}_+^{n},\sum _{i=1}^{n} x_{i}^m=1\}=0$. Therefore $\lambda \geq 0$. 

Again using theorem ($6(a)$) of \cite{Qi2005} we have $$|\lambda -l_{ii\dots i}|\leq  \sum_{\substack{i_{2}, i_{3}, \dots,  i_{m}=1, \\ \delta _{i,i_2,\dots ,i_m}=0}}^{n} |l_{ii_{2}i_{3}\dots i_{m}}| = \Delta,$$
i.e., $|\lambda| \leq 2 \Delta.$ Thus  $0\leq \lambda \leq 2\Delta$.
\end{proof}

\begin{thm}
 Let $G=(V, E)$ be a general hypergraph with $m.c.e(G) = m\geq 3$. Let $L$ be the Laplacian hypermatrix of $G$. Then
 \begin{enumerate}[(i)]
  \item $L$ has an $H$-eigenvalue 0 with eigenvector $(1, 1, \dots, 1)\in \mathbb{R}^n$ and  an $Z$-eigenvalue 0 with eigenvector $x=(\frac{1}{\sqrt{n}}, \frac{1}{\sqrt{n}}, \dots,  \frac{1}{\sqrt{n}})\in \mathbb{R}^{n}$. Moreover, $0$ is the unique $H^{++}$-eigenvalue of $L$.
  \item $\Delta$ is the largest $H^+$-eigenvalue of $L$.
   \item $(d(i), e^{(j)})$ is an $H$-eigenpair,  where $e^{(j)}\in\mathbb{R}^{n}$ and $e^{(j)}_{i}=1$ if $i=j$, otherwise 0.
  \item For a nonzero $x\in \mathbb{R}^n$ $(d(v_i), x)$ is an eigenpair if 
 $\sum_{e\in E_{i}}a_{G}^{e} x_{m-1}^{e/i}=0$.
 \end{enumerate}
\end{thm}
 \begin{proof}
 \begin{enumerate}[(i)]
 \item It is easy to check that $0$ is an $H$-eigenvalue with the eigenvctor $(1,1,1,\dots, 1)\in \mathbb{R}^{n}$ and $0$ is an $Z$-eigenvalue with the eigenvector  $x=(\frac{1}{\sqrt{n}}, \frac{1}{\sqrt{n}}, \dots,  \frac{1}{\sqrt{n}})\in \mathbb{R}^{n}$.
 Let $x$ is an $H^{++}$-eigenvector of $L$ with eigenvalue $\lambda$. By theorem (\ref {positivity}), $\lambda \geq 0.$ Suppose $x_{j}=\underset{i}{min}\{x_i\}$. Therefore $x_j$ is positive.
 Now,$$\lambda x_j^{m-1}=d(v_j)x_j^{m-1}-\sum _{e \in E,j\in e,|e|=s}\frac{s}{\alpha}\sum _{\substack{e\equiv \{i,i_2,\dots ,i_m\} \\ \text {as set, } i,i_2,\dots ,i_m=1}}^n x_{i_2} x_{i_3} \dots x_{i_m},$$ which implies that $$\lambda=d(v_j)-\sum _{e \in E,j\in e,|e|=s}\frac{s}{\alpha}\sum _{\substack{e\equiv \{i,i_2,\dots ,i_m\} \\ \text {as set, } i,i_2,\dots ,i_m=1}}^n \frac{x_{i_2}}{x_j}\frac{x_{i_3}}{x_j}\dots \frac{x_{i_m}}{x_j}.$$ Thus, $\lambda \leq d(v_j)-d(v_{j})=0$. Hence $\lambda=0$.
 \item Suppose  $\lambda$ is an $H^+$-eigenvalue with non-negative $H^+$-eigenvector, $x$ of $L$. Assume that $x_j>0$. Now, we have $$\lambda x_j^{m-1}=d(v_j)x_j^{m-1}-\sum _{e \in E,j\in e,|e|=s}\frac{s}{\alpha}\sum _{\substack{e\equiv \{i,i_2,\dots ,i_m\} \\ \text {as set, } i,i_2,\dots ,i_m=1}}^n x_{i_2} x_{i_3} \dots x_{i_m} \leq d(v_j)x_j^{m-1}.$$ Therefore $\lambda \leq d(v_j)\leq\Delta$. Thus, $\Delta$ is the largest $H^+$-eigenvalue of $L$. 
 \item Proof is obvious.
 \item  It is clear from the eigenvalue equation.
 \end{enumerate}
 \end{proof}

 Let $G=(V, E)$ be a general hypergraph and $m.c.e(G)=m$. The \textit {analytic connectivity}, $\alpha(G)$, of $G$ is defined as 
 $\alpha(G)= \underset{j=1, \dots,  n}{min}  min\{Lx^m | x \in \mathbb{R}_{+}^n,  \sum_{i=1}^{n}x_{i}^{m}=1, x_{j}=0\} $.
  \begin{thm}
 The general hypergraph $G=(V,E)$ with $m.c.e(G)\geq 3$ is connected if and only if $\alpha(G)>0.$
  \end{thm}
  \begin{proof}
  Suppose $G=(V,E)$ is not connected. Let $G_1=(V_1,E_1)$ be a component of $G$. Then there exists $j \in V\setminus V_1$. Let $x=\frac{1}{|V_1|^{\frac{1}{m}}}\sum _{i \in V_1}e^{(i)}$. Then $x$ is a feasible point. Therefore  $min\{Lx^m | x \in \mathbb{R}_{+}^n,  \sum_{i=1}^{n}x_{i}^{m}=1, x_{j}=0\}=0$, which implies $\alpha (G)=0$.
  
  Let $\alpha(G)=0$. Thus there exists $j$ such that $min\{Lx^m | x \in \mathbb{R}_{+}^n,  \sum_{i=1}^{n}x_{i}^{m}=1, x_{j}=0\}=0$. Suppose that $y$ is a minimizer of this minimization problem. Therefore $y_j=0$, $Ly^m=0$. By optimization theory, there exists a Lagrange multiplier $\mu$ such that for $i=1,2,\dots ,n$ and $i\neq j$, 
  either, $y_i=0$ and
  \begin{equation} 
  \frac{\partial}{\partial y_{i}}(Ly^m)\geq \mu \frac{\partial}{\partial y_i}(\sum_{i=1}^{n}y_{i}^{m}-1)
  \end{equation} 
  or, $y_i>0$ and
  \begin{equation}
   \frac{\partial}{\partial y_{i}}(Ly^m)=\mu \frac{\partial}{\partial y_i}(\sum_{i=1}^{n}y_{i}^{m}-1).
   \end{equation}
   In (2) and (3) $y \in \mathbb{R}_{+}^n,  \sum_{i=1}^{n}y_{i}^{m}=1, y_{j}=0$. Now, multiplying (2) and (3) by $y_i$ and summing them for $i=1, \dots,n$, we have $Ly^m=\mu (\sum_{i=1}^{n}y_{i}^{m})$. Thus $Ly^m=\mu$. Hence $\mu =0$. Therefore, for $i=1,2,\dots ,n$ and $i\neq j$, either $y_i=0$ or $\frac{\partial}{\partial y_i}(Ly^m)=0$. Hence, either $y_i=0$ or $d_i(y_i)^{m-1}-\sum_{i_{2}, i_{3}, \dots,  i_{m}=1}^{n} a_{i i_{2}i_{3}\dots i_{m}} y_{i_{2}}y_{i_{3}} \dots y_{i_{m}}=0$. Let $y_k=max\{y_i:i=1,2,\dots ,n\}$. Hence, we have $$d_k-\sum_{i_{2}, i_{3}, \dots,  i_{m}=1}^{n} a_{i i_{2}i_{3}\dots i_{m}}\frac{y_{i_2}}{y_k}\frac{y_{i_3}}{y_k}\dots \frac{y_{i_m}}{y_k}=0.$$ Again, we know that 
   $$d(v_{k})= \sum_{i_{2}, i_{3}, \dots,  i_{m}=1}^{n} a_{ii_{2}i_{3}\dots i_{m}}.$$ Therefore, $x_i=x_k$ as long as $i$ and $k$ belong to same edge. Thus,
   $x_i = x_k$ as long as $i$ and $k$ are in same component of $G$. Since $y_j=0$, we have, $j$ and $k$ are in different components of $G$. Hence, $G$ is not connected. This proves the theorem. 
  \end{proof}
 
%%%%%%%%%%%%%%%%%%%%%%%%%%
\subsection{Normalized Laplacian hypermatrix and eigenvalues}

Now, we define normalized Laplacian hypermatrix for a general hypergraph. For any graph, there are two ways to construct normalized Laplacian matrix (see  \cite{Banerjee2008} and \cite{chung1997}  for details)\footnote{These two matrices are similar, i.e., they have  same eigenvalues.}. Motivated by these two similar constructions, here, we also define the normalized Laplacian hypermatrix in two different ways and show that they are cospectral.
The first  definition is similar to the normalized Laplacian matrix defined in \cite{Banerjee2008}.

\begin{definition} 
 Let $G=(V, E)$ be a general hypergraph without any isolated vertex  where $V= \{ v_{1},  v_{2},  \dots,  v_{n} \}$ and $E = \{ e_{1},  e_{2},  \dots,  e_{k} \}$. Let $m.c.e(G)=m$.
 The normalized Laplacian hypermatrix $\mathcal{L} = (l_{i_1i_2\dots i_m})$, which is an $n$-dimensional $m$-th order hypermatrix, is defined as:   
 for any edge $e=\{ v_{l_{1}},  v_{l_{2}},  \dots, v_{l_{s}}\}\in E$ of cardinality $s \leq m$,   
  $$l_{ p_1p_{2}\dots p_{m}}=-\frac{s/\alpha}{d(v_{p_1})}, \text{ where }  \alpha =\sum_{\substack{k_{1}, k_{2}, \dots, k_{s} \geq 1, \\  \sum k_{i}=m}}  \frac{m!}{k_{1}! k_{2}!\dots k_{s}!}$$ and   $p_1, p_{2},  \dots,  p_{m}$  are chosen in all possible way  from $\{l_{1}, l_{2}, \dots, l_{s}\}$, such that, all $l_j $ occur at least once. All the diagonal entries are 1 and the rest are zero.
\end{definition}

Clearly, the hypermatrix $\mathit{A}=\mathcal{I}-\mathcal{L}$, which is known as normalized adjacency hypermatrix, is a stochastic tensor, that is, $\mathit{A}$ is non-negative and $ \sum_{i_{2}, \dots,  i_{m}=1}^{n} a_{i i_{2}\dots i_{m}}=1$, where $a_{i_1 i_{2}\dots i_{m}}$ is the $(i_1, i_2,\dots, i_m)$-th entry of  $\mathit{A}$.
%\begin{definition}
%A  non-negative  tensor $\mathcal{A}$ of order $m$ and dimension $n$ is called stochastic provide that 
%$$ \sum_{i_{2}, i_{3}, \dots,  i_{m}=1}^{n} a_{i i_{2}i_{3}\dots i_{m}}=1$$
%\end{definition}
 The different properties of a stochastic tensor are discussed in \cite{YangYang20101} and which can be used to study the hypermatrices 
 $\mathit{A}$ and $\mathcal{L}$.
Now, we define the normalized Laplacian hypermatrix of a general hypergraph as it is defined for a graph in \cite{chung1997}.

\begin{definition}
 Let $G=(V, E)$ be a general hypergraph without any isolated vertex,  where $V= \{ v_{1},  v_{2},  \dots,  v_{n} \}$ and $E = \{ e_{1},  e_{2},  \dots,  e_{k} \}$. Let $m.c.e(G)=m$.
 The normalized Laplacian hypermatrix $\mathfrak{L}= (l_{i_1i_2\dots i_m})$, which is an $n$-dimension $m$-th order symmetric  hypermatrix, is defined as:
for any edge $e=\{ v_{l_{1}},  v_{l_{2}},  \dots, v_{l_{s}}\}\in E$ of cardinality $s \leq m$,   
  $${l}_{ p_{1}p_{2}\dots p_{m}}=-\frac{s}{\alpha}\prod_{j=1}^m\frac{1}{\sqrt[m]{d(v_{p_{j}})}},  \text{ where } \alpha =\sum_{k_{1}, k_{2}, \dots, k_{s} \geq 1,  \sum k_{i}=m}  \frac{m!}{k_{1}! k_{2}!\dots k_{s}!}$$   and   $p_{1}, p_{2},  \dots,  p_{m}$  chosen in all possible way  from $\{l_{1}, l_{2}, \dots, l_{s}\}$ with at least once
  for each element of the set. The diagonal entries of  $\mathfrak{L}$ are $1$ and the rest of the positions are zero.
\end{definition}

\begin{thm}
  $\mathcal{L}$ and $\mathfrak{L}$ are co-spectral.
\end{thm}
\begin{proof}
In the lemma (\ref{co-spectral}) choose a diagonal matrix $D=(d_{ij})_{n \times n}$ where $d_{ii}=(d(v_i))^{1/m}$.
\end{proof}

\begin{thm}\label{spectra}
Let $G=(V, E)$ be a general hypergraph. Let $\mathcal{L}$, $\mathit{A}$ be the normalized Laplacian and normalized adjacency hypermatrices of $G$, respectively.
If $G$ has at least one edge, then $\lambda \in \sigma(\mathit{A})$ if and only if $(1-\lambda)\in \sigma(\mathcal{L})$, otherwise, $\sigma(\mathit{A})=\sigma(\mathcal{L})={0}$, where $\sigma(\mathcal{L})$ denotes the spectrum of $\mathcal{L}$.
\end{thm}
\begin{proof}
 Since, $\mathcal{L}=\mathcal{I}-\mathit{A}$ and $\lambda$ is the eigenvalue of $\mathit{A}$ \textit{iff} $\det(\mathit{A}- \lambda \mathcal{I})=0,$ thus, $\det(\mathcal{L}- (1-\lambda) \mathcal{I})=0$ implies $(1-\lambda)\in \sigma(\mathcal{L})$.
\end{proof}

\begin{thm}
 Let $G=(V, E)$ be a general hypergraph. Let $\mathcal{L}=(l_{i_{1}i_{2}\dots i_{m}})  \text{ where } 1\leq i_{1}, i_{2}, \dots, i_{m} \leq n,$ and $A$ be the normalized Laplacian and normalized adjacency hypermatrices of $G$, respectively, then
 \begin{enumerate}[(i)]
 \item $\rho(A)=1$.
  \item $0\leq\lambda(\mathcal{L})\leq 2$.
  \item 1 is the largest $H^+$-eigenvalue of $\mathcal{L}$.
  \item $0$ is an eigenvalue of $\mathcal{L}$ with the eigenvector $(1, 1, \dots, 1)$  and $0$ is an $Z$-eigenvalue with eigenvector  $x=(\frac{1}{\sqrt{n}}, \frac{1}{\sqrt{n}}, \dots,  \frac{1}{\sqrt{n}})\in \mathbb{R}^{n}$.
  \item $0$ is the unique $H^{++}$-eigenvalue of $\mathcal{L}$ .
 \end{enumerate} 
\end{thm}
\begin{proof}
\begin{enumerate}[(i)]
\item Since $A$ is stochastic tensor, it is obvious that the spectral radius of $A$ is 1. Moreover, $(1,1,\dots, 1)$ is an eigenvector with eigenvalue 1.
\item We know that spectral radius of $A$ is 1 and $\mathcal{L}=\mathcal{I}-A$. By theorem (\ref{spectra})  $\lambda \in \sigma(\mathit{A})$ if and only if $(1-\lambda)\in \sigma(\mathcal{L})$. Since 1 is an eigenvalue of $A$, thus, $\lambda \geq 0$.
Again, using theorem ($6(a)$) of \cite{Qi2005} we have $$|\lambda(\mathcal{L}) -1|\leq  \sum_{\substack{i_{2}, i_{3}, \dots,  i_{m}=1,\\ \delta _{i,i_2,\dots ,i_m}=0}}^{n} |l_{ii_{2}i_{3}\dots i_{m}}| = 1.$$
This implies $|\lambda(\mathcal{L})| \leq 2.$ Thus, we have $0\leq \lambda(\mathcal{L}) \leq 2$.
\item Suppose that $\lambda$ is an $H^+$-eigenvalue with non-negative $H^+$-eigenvector, $x$ of $\mathcal{L}$. Assume that $x_j>0$. Now, we have $$\lambda x_j^{m-1}=x_j^{m-1}-\frac{1}{d(v_j)}\sum _{e \in E,j\in e,|e|=s}\frac{s}{\alpha}\sum _{\substack{e\equiv\{i,i_2,\dots ,i_m\} \\ \text {as set, } i,i_2,\dots ,i_m=1}}^n x_{i_2} x_{i_3} \dots x_{i_m}.$$ Hence, $\lambda x_j^{m-1}\leq x_j^{m-1}$ implies $\lambda \leq 1$. Thus, 1 is the largest $H^+$-eigenvalue of $\mathcal{L}$.
\item It is easy to check that $0$ is an $H$-eigenvalue corresponding an eigenvector $(1,1,\dots, 1)\in \mathbb{R}^{n}$ and $0$ is an $Z$-eigenvalue with the eigenvector  $x=(\frac{1}{\sqrt{n}}, \frac{1}{\sqrt{n}}, \dots,  \frac{1}{\sqrt{n}})\in \mathbb{R}^{n}$.
\item Let $x$ is an $H^{++}$-eigenvector of $\mathcal{L}$ with eigenvalue $\lambda$. From the  part (ii) of this theorem we have $\lambda \geq 0.$ Suppose $x_{j}=\underset {i}{min}\{x_i\}$.
 Now,$$\lambda x_j^{m-1}=x_j^{m-1}-\frac{1}{d(v_j)}\sum _{e \in E,j\in e,|e|=s}\frac{s}{\alpha}\sum _{\substack{e\equiv\{i,i_2,\dots ,i_m\} \\ \text {as set} ,i,i_2,\dots ,i_m=1}}^n x_{i_2} x_{i_3} \dots x_{i_m},$$ which implies that $$\lambda=1-\frac{1}{d(v_j)}\sum _{e \in E,j\in e,|e|=s}\frac{s}{\alpha}\sum _{\substack{e\equiv\{i,i_2,\dots ,i_m\} \\ \text {as set, } i,i_2,\dots ,i_m=1}}^n \frac{x_{i_2}}{x_j}\frac{x_{i_3}}{x_j}\dots \frac{x_{i_m}}{x_j}.$$ Thus $\lambda \leq 1-1=0$. Hence $\lambda=0$.
\end{enumerate}
\end{proof}

\begin{thm}
 Let $G=(V, E)$ be a general hypergraph and $m.c.e(G)=m$. Let $\mathcal{L}$ be the normalized Laplacian hypermatrix of $G$ of order m and dimension n. Let $m(\lambda)$ be the algebraic multiplicity of $\lambda \in \sigma(\mathcal{L})$, then $\sum_{\lambda \in \sigma(\mathcal{L})} m(\lambda)\lambda =n(m-1)^{n-1}.$ 
\end{thm}
\begin{proof}
Since, for any tensor $\mathcal{T} = (t_{i_1i_2\dots i_m}), t_{i_1i_2\dots i_m} \in \mathbb{C},\text{  } 1\leq i_{1}, i_{2}, \dots, i_{m} \leq n,$ $$\sum _{\lambda \in \sigma (\mathcal{T})}m(\lambda)\lambda=(m-1)^{(n-1)} \sum _{i=1}^{n} t_{ii\dots i} \text{ (see \cite{Hu2013_3})}.$$
Hence, we have $\sum_{\lambda \in \sigma(\mathcal{L})} m(\lambda)\lambda =n(m-1)^{n-1}.$
\end{proof}

\begin{thm}
 Let $G=(V, E)$ be a general hypergraph and $A$ be any connectivity hypermatrix of $G$. If $G$ has $r\geq 1$ connected components, $G_{1}, G_{2}, \dots, G_{r}$, such that, $|V(G_{i})|=n_i>1$ and $m.c.e(G_i) = m.c.e(G)$ for each 
 $i\in\{1, 2, \dots, r \}$. Then, as sets, $\sigma(A)=\sigma(A_{1})\cup \sigma(A_{2})\cup\dots \cup \sigma(A_{r}),$ where $A_i$ is the connectivity  hypermatrix of $G_i$.
\end{thm}
\begin{proof}
Using corollary (4.2) of \cite {shaoshan2013} we get $$\phi _{A}(\lambda)=\prod _{i=1}^{r}(\phi _{A}(\lambda))^{(m-1)^{n-n_i}},$$ where $\phi _{A}(\lambda) $ is the characteristic polynomial of the tensor $A$. Therefore, $\sigma(A)=\sigma(A_{1})\cup \sigma(A_{2})\cup\dots \cup \sigma(A_{r})$.
\end{proof}

\section{Discussion and conclusion}
Here, we propose a mathematical framework to construct connectivity matrices for a general hypergraph and also study the eigenvalues of adjacency hypermatrix, Laplacian hypermatrix, normalized Laplacian hypermatrix. This connectivity hypermatrix reconstruction  can be used for further development of  spectral hypergraph theory in many aspects, but, this may not be  quite useful to study  dynamics on  hypergraphs. 

\section*{Acknowledgements}
The authors are thankful to Mithun Mukherjee and Swarnendu Datta for  fruitful discussions. Financial support from Council of Scientific and Industrial Research,  India, Grant no-09/921(0113)/2014-EMR-I is sincerely acknowledged by Bibhash Mondal.

\end{document}